\documentclass[a4paper,12pt]{article}
\usepackage{latexsym,amssymb}
\usepackage{amsmath}
\usepackage[mathscr]{eucal}
\usepackage{epic,eepic}
\usepackage{color}
\usepackage{enumerate}

\usepackage{amsthm}
\newtheorem{theorem}{Theorem}[section]
\newtheorem{proposition}[theorem]{Proposition}
\newtheorem{lemma}[theorem]{Lemma}

\theoremstyle{definition}
\newtheorem{definition}[theorem]{Definition}

\theoremstyle{remark}
\newtheorem{remark}[theorem]{Remark}
\newtheorem{example}[theorem]{Example}

\makeatletter

\@addtoreset{equation}{section}
\makeatother

\newcommand{\R}{\mathbb{R}}
\newcommand{\E}{\mathbb{E}}
\newcommand{\sd}{\mathsf{d}}
\newcommand{\sJ}{\mathsf{J}}
\newcommand{\cB}{\mathcal{B}}
\newcommand{\cP}{\mathcal{P}}
\newcommand{\bV}{\mathbf{V}}

\newcommand{\ve}{\varepsilon}

\newcommand{\lra}{\longrightarrow}

\newcommand{\CAT}{\mathrm{CAT}}

\newcommand{\diam}{\mathrm{diam}}
\newcommand{\supp}{\mathrm{supp}}

\def\argmin{\mathop{\mathrm{arg\, min}}}

\begin{document}

\title{Barycenters and a law of large numbers\\ in Gromov hyperbolic spaces}

\author{Shin-ichi OHTA\thanks{
Department of Mathematics, Osaka University, Osaka 560-0043, Japan
({\sf s.ohta@math.sci.osaka-u.ac.jp}),
RIKEN Center for Advanced Intelligence Project (AIP),
1-4-1 Nihonbashi, Tokyo 103-0027, Japan}}

\date{\empty}
\maketitle

\begin{abstract}
We investigate barycenters of probability measures on Gromov hyperbolic spaces,
toward development of convex optimization in this class of metric spaces.
We establish a contraction property (the Wasserstein distance between probability measures
provides an upper bound of the distance between their barycenters),
a deterministic approximation of barycenters of uniform distributions on finite points,
and a kind of law of large numbers.
These generalize the corresponding results on CAT(0)-spaces,
up to additional terms depending on the hyperbolicity constant.
\medskip

\end{abstract}


\section{Introduction}

This article is a continuation of \cite{Ohyp} in which we studied discrete-time gradient flows
for geodesically convex functions on (geodesic, proper) Gromov hyperbolic spaces.
The theory of gradient flows for convex functions possesses fundamental importance
in analysis, geometry and optimization theory, and has been well investigated
in some classes of ``Riemannian'' metric spaces including \emph{$\CAT(0)$-spaces}
(nonpositively curved metric spaces in the sense of triangle comparison; we refer to \cite{Babook}).
For ``non-Riemannian'' metric spaces such as normed spaces and Finsler manifolds,
however, much less is known and, in fact, there is a large gap between
properties of gradient flows in Riemannian and Finsler manifolds
(see \cite{Ohyp,OSnc} for further discussions).

Intending to develop optimization theory in possibly non-Riemannian spaces,
in \cite{Ohyp} we studied discrete-time gradient flows
for geodesically convex functions on Gromov hyperbolic spaces
and showed some contraction estimates.
Gromov hyperbolic spaces are metric spaces negatively curved in large-scale,
and it is known that some non-Riemannian Finsler manifolds can be Gromov hyperbolic
(see Example~\ref{ex:Ghyp}).
The class of geodesically convex functions seems, however,
restrictive when one has in mind the local flexibility of the Gromov hyperbolicity condition.
Thereby, it is desirable to generalize the theory of gradient flows
to a wider class of ``convex functions'' (we refer to \cite[\S 3.4]{Ohyp} for related discussions).
As an initial step toward such a generalization, in this article,
we study the (squared) distance function on a Gromov hyperbolic space,
which should be included in the class of generalized convex functions
(in view of Lemmas~\ref{lm:CAT0t}, \ref{lm:Buse}).

Given a probability measure $\mu$ on a metric space $(X,d)$ with finite second moment,
its \emph{barycenter} is defined as a minimizer of the function $x \mapsto \int_X d^2(x,z) \,\mu(dz)$.
If the (squared) distance function is sufficiently convex,
then barycenters enjoy a number of fine properties.
Specifically, in a $\CAT(0)$-space $(X,d)$ for which $d^2$ is strictly convex by definition,
every $\mu$ admits a unique barycenter $\beta_{\mu} \in X$
and we have a \emph{contraction property} $d(\beta_{\mu},\beta_{\nu}) \le W_1(\mu,\nu)$
in terms of the $L^1$-Wasserstein distance $W_1$ (see \cite{St1}).
Moreover, a kind of \emph{law of large numbers} providing the almost sure convergence
to barycenters by recursive applications of the proximal (resolvent) operator was established in \cite{St1}
(we refer to \cite{OP1,Yo3} for some generalizations).

A metric space $(X,d)$ is said to be \emph{Gromov hyperbolic} (attributed to \cite{Gr})
if it is \emph{$\delta$-hyperbolic} for some $\delta \ge 0$ in the sense that
\[ (x|z)_p \ge \min\{ (x|y)_p,(y|z)_p \} -\delta \]
holds for all $p,x,y,z \in X$, where
\[ (x|y)_p :=\frac{1}{2} \big\{ d(p,x)+d(p,y)-d(x,y) \big\} \]
is the \emph{Gromov product}.
This is a large-scale notion of negative curvature and hence, on the one hand,
it is natural to expect some variants of the aforementioned results in $\CAT(0)$-spaces.
On the other hand, since the Gromov hyperbolicity provides no local control
(up to the hyperbolicity constant $\delta$),
one cannot expect very sharp estimates as in the case of $\CAT(0)$-spaces.
Accordingly, our results will have additional terms
(compared with the case of $\CAT(0)$-spaces) those tend to $0$ as $\delta \to 0$.

For a probability measure $\mu$ on a (geodesic) $\delta$-hyperbolic space $(X,d)$,
its barycenter is not unique but lives in a bounded set
(whose diameter tends to $0$ as $\delta \to 0$; see Proposition~\ref{pr:bary}).
For this reason, we introduce the set
\[ \cB(\mu,\ve) :=\bigg\{ x \in X \,\bigg|\,
 \int_X d^2(x,z) \,\mu(dz) \le \inf_{y \in X} \int_X d^2(y,z) \,\mu(dz) +\ve \bigg\} \]
for $\ve \ge 0$, and call it a \emph{barycentric set}
(here we consider only probability measures of finite second moment for simplicity).
Then, we show a contraction property of the form
\[ d(x,y) \le W_1(\mu,\nu) +\sqrt{6\ve} +O(\delta^{1/4}) \]
for $x \in \cB(\mu,\ve)$ and $y \in \cB(\nu,\ve)$; see Theorem~\ref{th:Wcont} for the precise statement.

How to find (or approximate) a barycenter of a given probability measure $\mu$ is a fundamental problem.
In this respect, we show the following law of large numbers
(see Theorem~\ref{th:LLN} for the precise statement):
Given a sequence $(Z_i)_{i \ge 1}$ of independent, identically distributed random variables
with distribution $\mu$ and an arbitrary initial point $S_0 \in X$,
we consider a sequence $(S_k)_{k \ge 0}$ recursively chosen as $S_{k+1}=\gamma((2\tau+1)^{-1})$
for a minimal geodesic $\gamma:[0,1] \lra X$ from $Z_{k+1}$ to $S_k$.
Then, for any $\ve>0$, we have
\[ \E\big[ d^2(p,S_{k_0}) \big] \le \ve +O(\sqrt{\delta}) \]
for some $k_0 \le C/(\ve\sqrt{\delta})$ by taking $\tau$ proportional to $\sqrt{\delta}$,
where $p \in \cB(\mu,0)$ is a barycenter of $\mu$.
We remark that the construction of $S_{k+1}$ from $S_k$
is written by the proximal operator $S_{k+1} \in \sJ_{\tau}^{f_k}(S_k)$
for the distance function $f_k=d^2(Z_{k+1},\cdot)$ (see \eqref{eq:sJ}),
and that such an operation is meaningful only for large $\tau$ compared with $\delta$
due to the local flexibility of $\delta$-hyperbolic spaces.
By a similar analysis, we prove in Theorem~\ref{th:nodice}
a \emph{deterministic approximation} of barycenters of uniform distributions on finite points,
as a generalization of \cite{LP} in $\CAT(0)$-spaces
(we refer to \cite{Ba1,Ba2,Ho,OP1} for some related results).

We briefly mention related results on \emph{hyperbolic groups}
(discrete groups whose Cayley graphs are Gromov hyperbolic).
Laws of large numbers can be formulated in terms of the behavior of the distance function
$d(g_1 g_2 \cdots g_k(x_0),x_0)$ for a sequence $(g_i)_{i \ge 1}$
of independent, identically distributed random variables taking values
in the group of isometric transformations acting on a metric space $(X,d)$
and a base point $x_0 \in X$.
Indeed, Kingman's subadditive ergodic theorem \cite{Ki1,Ki2} ensures that
\[ \lim_{k \to \infty} \frac{d(g_1 g_2 \cdots g_k(x_0),x_0)}{k} \]
exists almost surely and is constant almost everywhere.
We refer to \cite{KL} and the references therein for more details as well as
refined results including the case of hyperbolic groups.
In this context, \emph{central limit theorems} in hyperbolic groups were also established in \cite{BQ,Bj}.
Our law of large numbers (Theorem~\ref{th:LLN})
associated with a probability measure $\mu$ on a Gromov hyperbolic space $X$
is concerned with a more general setting and provides a direct approximation of barycenters.
In this generality, it is difficult even to formulate central limit theorems.

This article is organized as follows.
In Section~\ref{sc:pre}, we review the basics of Gromov hyperbolic spaces
and some facts on barycenters in $\CAT(0)$-spaces.
In Section~\ref{sc:bary}, we introduce and analyze barycentric sets for probability measures
on Gromov hyperbolic spaces.
Then we discuss the Wasserstein contraction property, a deterministic approximation,
and a law of large numbers in Sections~\ref{sc:Wcont}, \ref{sc:nodice}, and \ref{sc:LLN},
respectively.

\section{Preliminaries}\label{sc:pre}

We review the basics of Gromov hyperbolic spaces,
as well as some facts on barycenters in $\CAT(0)$-spaces related to our results.
For $a,b \in \R$, we set $a \wedge b:=\min\{a,b\}$ and $a \vee b:=\max\{a,b\}$.

\subsection{Gromov hyperbolic spaces}\label{ssc:Ghyp}

Besides the original paper \cite{Gr}, we refer to \cite{Bo,BH,DSU,Ro,Va}
for the basics and various applications of the Gromov hyperbolicity.

Let $(X,d)$ be a metric space.
For three points $x,y,z \in X$, we define the \emph{Gromov product} $(y|z)_x$ by
\[ (y|z)_x :=\frac{1}{2} \big\{ d(x,y) +d(x,z) -d(y,z) \big\}. \]
Observe from the triangle inequality that
$0 \le (y|z)_x \le d(x,y) \wedge d(x,z)$.
In the Euclidean plane $\R^2$,
$(y|z)_x$ is understood as the distance from $x$
to the intersection of the triangle $\triangle xyz$ and its inscribed circle.
If $x,y,z$ are the endpoints of a tripod,
then $(y|z)_x$ coincides with the distance from $x$ to the branching point.

\begin{definition}[Gromov hyperbolic spaces]\label{df:Ghyp}
A metric space $(X,d)$ is said to be \emph{$\delta$-hyperbolic} for $\delta \ge 0$ if
\begin{equation}\label{eq:d-hyp}
(x|z)_p \ge (x|y)_p \wedge (y|z)_p -\delta
\end{equation}
holds for all $p,x,y,z \in X$.
We say that $(X,d)$ is \emph{Gromov hyperbolic}
if it is $\delta$-hyperbolic for some $\delta \ge 0$.
\end{definition}

Since the Gromov product does not exceed the diameter $\diam(X):=\sup_{x,y \in X}d(x,y)$,
if $\diam(X) \le \delta$, then $(X,d)$ is $\delta$-hyperbolic.
This also means that the local structure of $X$ (up to size $\delta$)
is not influential in the $\delta$-hyperbolicity.
Another fact worth mentioning is that, if \eqref{eq:d-hyp} holds for some $p \in X$ and all $x,y,z \in X$,
then $(X,d)$ is $2\delta$-hyperbolic (see \cite[Corollary~1.1.B]{Gr}).

The Gromov hyperbolicity can be regarded as a large-scale notion of negative (sectional) curvature.
We recall some fundamental examples (see also \cite[\S 1]{Gr}).

\begin{example}\label{ex:Ghyp}
\begin{enumerate}[(a)]
\item
Complete, simply connected Riemannian manifolds of sectional curvature $\le -1$
(or, more generally, $\CAT(-1)$-spaces)
are Gromov hyperbolic (see, for example, \cite[\S 1.5]{Gr}).

\item
An important difference between the class of $\CAT(-1)$-spaces and that of Gromov hyperbolic spaces
is that the latter admits non-Riemannian Finsler manifolds.
For instance, \emph{Hilbert geometry} on a bounded convex domain in the Euclidean space
is Gromov hyperbolic under mild convexity and smoothness conditions
(see \cite{KN}, \cite[\S 6.5]{Obook}).

\item
The definition \eqref{eq:d-hyp} makes sense for discrete spaces.
In fact, the Gromov hyperbolicity has found rich applications in group theory
(a discrete group whose Cayley graph satisfies the Gromov hyperbolicity
is called a \emph{hyperbolic group}; we refer to \cite{Bo,Gr}, \cite[Part~III]{BH}).
In the sequel, however, we do not consider discrete spaces.

\item
Assume that a metric space $(X,d_X)$ admits a map $\phi:T \lra X$ from a tree $(T,d_T)$ such that
$d_X(\phi(a),\phi(b))=d_T(a,b)$ for all $a,b \in T$
and the $\delta$-neighborhood $B(\phi(T),\delta)$ of $\phi(T)$ covers $X$.
Then, since $(T,d_T)$ is $0$-hyperbolic, we can easily see that $(X,d_X)$ is $6\delta$-hyperbolic.
\end{enumerate}
\end{example}

We call $(X,d)$ a \emph{geodesic space}
if any two points $x,y \in X$ are connected by a \emph{minimal geodesic}
$\gamma:[0,1] \lra X$ satisfying $\gamma(0)=x$, $\gamma(1)=y$, and
$d(\gamma(s),\gamma(t))=|s-t| \cdot d(x,y)$ for all $s,t \in [0,1]$.
In this case, there are several characterizations of the Gromov hyperbolicity
(see, for example, \cite[\S 6]{Gr}, \cite[\S III.H.1]{BH}).
We also remark that, by \cite[Theorem~4.1]{BS},
every $\delta$-hyperbolic metric space can be isometrically embedded
into a complete, geodesic $\delta$-hyperbolic space.

\subsection{$\CAT(0)$-spaces}\label{ssc:CAT0}

A geodesic space $(X,d)$ is called a \emph{$\CAT(0)$-space}
if, for any $x,y,z \in X$ and any minimal geodesic $\gamma:[0,1] \lra X$ from $x$ to $y$,
\begin{equation}\label{eq:CAT}
d^2\big( z,\gamma(t) \big) \le (1-t)d^2(z,x) +td^2(z,y) -(1-t)td^2(x,y)
\end{equation}
holds for all $t \in (0,1)$.
A complete, simply connected Riemannian manifold is a $\CAT(0)$-space
if and only if its sectional curvature is nonpositive everywhere.
Moreover, there are a number of rich classes of non-smooth $\CAT(0)$-spaces
such as Euclidean buildings, trees, phylogenetic tree spaces,
and the orthoscheme complexes of modular lattices (see \cite{Babook,BHV,CCHO,Jobook}).

The $\CAT(0)$-inequality \eqref{eq:CAT} can be regarded as the uniform strict convexity
of the squared distance function $d^2(z,\cdot)$,
and such a convexity is known to be quite useful to study barycenters.
In fact, as we mentioned in the introduction,
every probability measure $\mu$ on $X$ with finite second (or first) moment
admits a unique barycenter $\beta_{\mu} \in X$,
and the contraction property $d(\beta_{\mu},\beta_{\nu}) \le W_1(\mu,\nu)$ holds.
Moreover, a law of large numbers was established in \cite{St1},
followed by many variants and generalizations \cite{Ba2,Ho,LP,OP1,Yo3}
(see also \cite{EH,Na} for related works on a different notion of barycenter).

\section{Barycentric sets}\label{sc:bary}

Henceforth, throughout the article, let $(X,d)$ be a geodesic $\delta$-hyperbolic space.
We first recall the Wasserstein distance between probability measures
(we refer to \cite{Vi} for further reading).

Denote by $\cP^p(X)$ ($p \in [1,\infty)$) the set of Borel probability measures on $X$
of finite $p$-th moment, that is, $\mu \in \cP^p(X)$ if
\[ \int_X d^p(x_0,x) \,\mu(dx) <\infty \]
for some (and hence any) $x_0 \in X$.
For $\mu,\nu \in \cP^p(X)$, the \emph{$L^p$-Wasserstein distance}
(or the \emph{Kantorovich distance}) between $\mu$ and $\nu$ is defined by
\[ W_p(\mu,\nu) :=\inf_{\pi} \bigg( \int_{X \times X} d^p(x,y) \,\pi(dx\,dy) \bigg)^{1/p}, \]
where $\pi$ runs over all \emph{couplings} of $(\mu,\nu)$
(namely probability measures on $X \times X$ with marginals $\mu$ and $\nu$).
A coupling $\pi$ attaining the above infimum is called an \emph{$L^p$-optimal coupling} of $(\mu,\nu)$.
Observe that
\[ W_p(\delta_x,\mu) =\bigg( \int_X d^p(x,y) \,\mu(dy) \bigg)^{1/p} \]
for all $x \in X$ and $\mu \in \cP^p(X)$, where $\delta_x$ denotes the Dirac mass at $x$.

Note that $\cP^2(X) \subset \cP^1(X)$ by the H\"older (or Cauchy--Schwarz) inequality.
According to \cite{St1}, we will consider barycenters of probability measures
not only in $\cP^2(X)$ but also in $\cP^1(X)$.
Fix an arbitrary point $x_0 \in X$.
For $\mu \in \cP^1(X)$, we define
\begin{equation}\label{eq:Vx}
\bV\!_{x_0}(\mu) :=\inf_{x \in X} \int_X \big\{ d^2(x,z) -d^2(x_0,z) \big\} \,\mu(dz).
\end{equation}
We remark that the integral above is well-defined since
\begin{align*}
\int_X |d^2(x,z) -d^2(x_0,z)| \,\mu(dz)
&\le \int_X d(x,x_0) \big\{ d(x,z)+d(x_0,z) \big\} \,\mu(dz) \\
&= d(x,x_0) \big\{ W_1(\delta_x,\mu) +W_1(\delta_{x_0},\mu) \big\}
 <\infty
\end{align*}
by the triangle inequality.

In a complete $\CAT(0)$-space,
every $\mu \in \cP^1(X)$ admits a unique point $x \in X$ attaining the infimum in \eqref{eq:Vx}.
Such a point $x$ is independent of the choice of $x_0$
and called the \emph{barycenter} of (or the \emph{center of mass} for) $\mu$
(see \cite[Proposition~4.3]{St1}, and \cite{Yo1} for the case of $\CAT(1)$-spaces of small radii).
More precisely, what we consider is an $L^2$-barycenter involving the squared distance.
We refer to \cite{Af,Yo2} for related works on $L^p$-barycenters.

In a $\delta$-hyperbolic space, however, we do not have a unique barycenter.
For this reason, we introduce the following set:
\begin{equation}\label{eq:bary}
\cB(\mu,\ve) :=\bigg\{ x \in X \,\bigg|\,
 \int_X \big\{ d^2(x,z) -d^2(x_0,z) \big\} \,\mu(dz) \le \bV\!_{x_0}(\mu) +\ve \bigg\}
\end{equation}
for $\mu \in \cP^1(X)$ and $\ve \ge 0$.
We shall call $\cB(\mu,\ve)$ a \emph{barycentric set}.
We remark that $\cB(\mu,\ve)$ is independent of the choice of $x_0$.
Note also that $\cB(\mu,0)$ may be empty (unless $X$ is proper),
while $\cB(\mu,\ve) \neq \emptyset$ for any $\ve>0$.

Our goal in this section is to estimate the diameter of $\cB(\mu,\ve)$
in terms of $\ve$ and $\delta$.
To this end, we first generalize the $\CAT(0)$-inequality \eqref{eq:CAT} to $\delta$-hyperbolic spaces,
with an inevitable additional term depending on $\delta$.

\begin{lemma}[$\CAT(0)+\delta$]\label{lm:CAT0}
For any $x,y,z \in X$ and any midpoint $w$ between $x$ and $y$,
we have
\begin{equation}\label{eq:CAT0}
d^2(z,w) \le \frac{d^2(z,x)}{2} +\frac{d^2(z,y)}{2} -\frac{d^2(x,y)}{4}
 +2\delta \big\{ d(z,x) +d(z,y) \big\} +4\delta^2.
\end{equation}
\end{lemma}

\begin{proof}
Since $w$ is a midpoint of $x$ and $y$ (namely $d(x,w)=d(w,y)=d(x,y)/2$),
we have $(x|y)_w=0$.
Then the $\delta$-hyperbolicity \eqref{eq:d-hyp} implies
\[ 0 \ge (x|z)_w \wedge (z|y)_w -\delta
 =\frac{1}{2} \bigg\{ d(z,w) +\frac{d(x,y)}{2} -d(z,x) \vee d(z,y) \bigg\} -\delta. \]
Hence,
\[ d^2(z,w) \le \bigg( d(z,x) \vee d(z,y) -\frac{d(x,y)}{2} +2\delta \bigg)^2. \]
Setting $a=d(z,x) \vee d(z,y)$ and $b=d(z,x) \wedge d(z,y)$,
we observe that
\begin{align*}
&\bigg( a-\frac{d(x,y)}{2} \bigg)^2 \\
&= a^2 +\frac{d^2(x,y)}{4} -ad(x,y) \\
&= \frac{a^2 +b^2}{2} -\frac{d^2(x,y)}{4}
 -\frac{1}{2} \big\{ d(x,y)-(a-b) \big\} \big\{ (a+b) -d(x,y) \big\} \\
&\le \frac{d^2(z,x)}{2} +\frac{d^2(z,y)}{2} -\frac{d^2(x,y)}{4},
\end{align*}
since $a+b \ge d(x,y) \ge a-b$ by the triangle inequality.
Moreover, we find
\[ 0 \le a -\frac{d(x,y)}{2}
 \le a -\frac{a-b}{2} =\frac{d(z,x) +d(z,y)}{2}. \]
Combining these, we obtain
\begin{align*}
d^2(z,w)
&\le \bigg( a-\frac{d(x,y)}{2} \bigg)^2 +4\delta \bigg( a-\frac{d(x,y)}{2} \bigg) +4\delta^2 \\
&\le \frac{d^2(z,x)}{2} +\frac{d^2(z,y)}{2} -\frac{d^2(x,y)}{4}
 +2\delta \big\{ d(z,x) +d(z,y) \big\} +4\delta^2
\end{align*}
as desired.
\end{proof}

We also present the corresponding inequality for general intermediate points
between $x$ and $y$.

\begin{lemma}[General intermediate points]\label{lm:CAT0t}
For any $x,y,z \in X$ and any minimal geodesic $\gamma:[0,1] \lra X$ from $x$ to $y$,
we have
\begin{equation}\label{eq:CAT0t}
d^2\big( z,\gamma(t) \big) \le (1-t)d^2(z,x) +td^2(z,y) -(1-t)td^2(x,y)
 +4\delta \big( d(z,x) \vee d(z,y) \big) +4\delta^2
\end{equation}
for all $t \in (0,1)$.
\end{lemma}

\begin{proof}
Put $w=\gamma(t)$, then we have $(x|y)_w=0$ again.
Thus, it follows from the $\delta$-hyperbolicity \eqref{eq:d-hyp} that
\begin{align*}
0 &\ge (x|z)_w \wedge (z|y)_w -\delta \\
&= \frac{1}{2} \Big\{ d(z,w) +\big( td(x,y) -d(z,x) \big) \wedge \big( (1-t)d(x,y) -d(z,y) \big) \Big\} -\delta,
\end{align*}
and hence
\[ d^2(z,w) \le
 \Big\{ \big( d(z,x)-td(x,y) \big) \vee \big( d(z,y)-(1-t)d(x,y) \big) +2\delta \Big\}^2. \]
Now, we claim that
\begin{align}
&\Big\{ \big( d(z,x)-td(x,y) \big) \vee \big( d(z,y)-(1-t)d(x,y) \big) \Big\}^2 \nonumber\\
&\le (1-t)d^2(z,x) +td^2(z,y) -(1-t)td^2(x,y) \label{eq:claim}
\end{align}
holds.
In fact, the inequality
\[ \big( d(z,x)-td(x,y) \big)^2 \le (1-t)d^2(z,x) +td^2(z,y) -(1-t)td^2(x,y) \]
can be rearranged as
\[ td^2(z,x) -2td(z,x)d(x,y) +td^2(x,y) \le td^2(z,y), \]
which holds true by the triangle inequality.
We can similarly show
\[ \big( d(z,y)-(1-t)d(x,y) \big)^2 \le (1-t)d^2(z,x) +td^2(z,y) -(1-t)td^2(x,y), \]
thereby we obtain \eqref{eq:claim}.
Therefore, we deduce that
\[ d^2(z,w) \le (1-t)d^2(z,x) +td^2(z,y) -(1-t)td^2(x,y)
 +4\delta \big( d(z,x) \vee d(z,y) \big) +4\delta^2. \]
\end{proof}

Note that, in the case of $\delta=0$,
\eqref{eq:CAT0t} boils down to the $\CAT(0)$-inequality \eqref{eq:CAT}.
For $\delta=0$, moreover, one can infer \eqref{eq:CAT0t} from \eqref{eq:CAT0} 
by the standard subdivision argument
(see, for example, (ii) $\Rightarrow$ (iii) of \cite[Theorem~1.3.3]{Babook}).
For $\delta>0$, however, iterating subdivisions makes the additional term (depending on $\delta$) diverge,
thereby we gave a direct argument to prove Lemma~\ref{lm:CAT0t}.
We also remark that \eqref{eq:CAT0t} is not meaningful for $t$ close to $0$ or $1$,
since then the triangle inequality could give a better estimate.

We are ready to estimate the diameter of barycentric sets $\cB(\mu,\ve)$ defined in \eqref{eq:bary}.
Recall that $x_0 \in X$ is an arbitrary point fixed at the beginning of this section.

\begin{proposition}[Diameter of $\cB(\mu,\ve)$]\label{pr:bary}
For any $\mu \in \cP^1(X)$, $x,y \in X$ and any midpoint $w$ between $x$ and $y$,
we have
\begin{align*}
\int_X \big\{ d^2(w,z)-d^2(x_0,z) \big\} \,\mu(dz)
&\le \int_X \bigg\{ \frac{d^2(x,z)}{2} +\frac{d^2(y,z)}{2} -d^2(x_0,z) \bigg\} \,\mu(dz) \\
&\quad -\frac{d^2(x,y)}{4}
 +2\delta \big\{ W_1(\delta_x,\mu) +W_1(\delta_y,\mu) \big\} +4\delta^2.
\end{align*}
In particular, for any $x \in \cB(\mu,\ve_1)$ and $y \in \cB(\mu,\ve_2)$ with $\ve_1,\ve_2 \ge 0$,
we have
\begin{equation}\label{eq:vineq}
d(x,y) \le
 \sqrt{ 8\delta \big\{ W_1(\delta_x,\mu) +W_1(\delta_y,\mu) \big\} +16\delta^2 +2(\ve_1 +\ve_2)}.
\end{equation}
\end{proposition}

\begin{proof}
The first assertion is shown by integrating \eqref{eq:CAT0} in $z$ with respect to $\mu$.
Then, when $x \in \cB(\mu,\ve_1)$ and $y \in \cB(\mu,\ve_2)$, we find
\begin{align*}
\bV\!_{x_0}(\mu)
&\le \int_X \big\{ d^2(w,z)-d^2(x_0,z) \big\} \,\mu(dz) \\
&\le \bV\!_{x_0}(\mu) +\frac{\ve_1 +\ve_2}{2} -\frac{d^2(x,y)}{4}
 +2\delta \big\{ W_1(\delta_x,\mu) +W_1(\delta_y,\mu) \big\} +4\delta^2.
\end{align*}
Therefore, we obtain
\[ d^2(x,y) \le 8\delta \big\{ W_1(\delta_x,\mu) +W_1(\delta_y,\mu) \big\} +16\delta^2 +2(\ve_1 +\ve_2). \]
This completes the proof.
\end{proof}

The second assertion \eqref{eq:vineq} (with $\ve_2 =0$) can be regarded as a generalization of
the \emph{variance inequality} (see \cite[Proposition~4.4]{St1};
the reverse inequality under lower curvature bounds can be found in \cite{Obary,St2}).
Note that, when we are interested in the case of $\ve_1 =\ve_2 =0$ (the set of barycenters),
\eqref{eq:vineq} implies $\diam(\cB(\mu,0)) \le O(\sqrt{\delta})$ as $\delta \to 0$.

\begin{remark}[When $\mu \in \cP^2(X)$]\label{eq:L^2}
In the case of $\mu \in \cP^2(X)$, instead of $\bV\!_{x_0}(\mu)$ as in \eqref{eq:Vx},
we can directly consider
\[ \inf_{x \in X} W_2^2(\delta_x,\mu) =\inf_{x \in X} \int_X d^2(x,z) \,\mu(dz), \]
which is called the \emph{variance} of $\mu$.
One can simply write down the first assertion of Proposition~\ref{pr:bary} as
\[ W_2^2(\delta_w,\mu)
 \le \frac{W_2^2(\delta_x,\mu)}{2} +\frac{W_2^2(\delta_y,\mu)}{2} -\frac{d^2(x,y)}{4}
 +2\delta \big\{ W_1(\delta_x,\mu) +W_1(\delta_y,\mu) \big\} +4\delta^2 \]
and, if $W_2^2(\delta_x,\mu) \le \inf_{p \in X} W_2^2(\delta_p,\mu) +\ve_1$
and $W_2^2(\delta_y,\mu) \le \inf_{p \in X} W_2^2(\delta_p,\mu) +\ve_2$,
then we have \eqref{eq:vineq}.
\end{remark}

\section{Wasserstein contraction property}\label{sc:Wcont}

We next consider a contraction property in terms of the Wasserstein distance,
which in the case of (complete) $\CAT(0)$-spaces means that
\[ d(\beta_{\mu},\beta_{\nu}) \le W_1(\mu,\nu) \]
holds for $\mu,\nu \in \cP^1(X)$, where $\beta_{\mu}$ and $\beta_{\nu}$
are the (unique) barycenters of $\mu$ and $\nu$, respectively
(see \cite[Theorem~6.3]{St1}).
In the current setting of $\delta$-hyperbolic spaces,
we shall estimate the distance between points in the barycentric sets.

We begin with a generalization of the \emph{Busemann nonpositive curvature}
(Busemann NPC for short).
We say that a geodesic space $(Z,d_Z)$ has the Busemann NPC (or $(Z,d_Z)$ is \emph{convex})
if, for any geodesics $\xi,\zeta:[0,1] \lra Z$ with $\xi(0)=\zeta(0)$,
we have $d_Z(\xi(t),\zeta(t)) \le td_Z(\xi(1),\zeta(1))$ for all $t \in [0,1]$.
Then, by the triangle inequality,
\[ d_Z\big( \xi(t),\zeta(t) \big) \le (1-t)d_Z\big( \xi(0),\zeta(0) \big) +td_Z\big( \xi(1),\zeta(1) \big) \]
holds for any geodesics $\xi,\zeta:[0,1] \lra Z$ (regardless of whether $\xi(0)=\zeta(0)$ or not).
We refer to \cite{BH,Jobook} for further reading.

\begin{remark}[Busemann NPC versus $\CAT(0)$]\label{rm:Buse}
In his celebrated paper \cite{Bu}, Busemann showed that
a complete, simply connected Riemannian manifold has the Busemann NPC
if and only if its sectional curvature is nonpositive everywhere.
Nonetheless, in general, the Busemann NPC is a weaker condition than the $\CAT(0)$-property.
On the one hand, it is easily seen that $\CAT(0)$-spaces have the Busemann NPC.
On the other hand, every strictly convex Banach space has the Busemann NPC,
whereas it is a $\CAT(0)$-space if and only if it is a Hilbert space.
\end{remark}

Recall that $(X,d)$ will always denote a geodesic $\delta$-hyperbolic space.
It is known that $\delta$-hyperbolic spaces have the Busemann NPC
up to an additive constant depending only on $\delta$ (see \cite[\S 7.4]{Gr}).
We give an outline of the proof for completeness.

\begin{lemma}[Busemann NPC${}+\delta$]\label{lm:Buse}
Let $x,y,p,q \in X$.
For any minimal geodesics $\gamma:[0,1] \lra X$ from $x$ to $p$
and $\eta:[0,1] \lra X$ from $y$ to $q$, we have
\begin{equation}\label{eq:Buse}
d\big( \gamma(t),\eta(t) \big) \le (1-t)d(x,y) +td(p,q) +8\delta
\end{equation}
for all $t \in (0,1)$.
\end{lemma}

\begin{proof}
Let $\xi :[0,1] \lra X$ and $\zeta:[0,1] \lra X$ be minimal geodesics
from $x$ to $q$ and from $p$ to $q$, respectively.
Denote by $\triangle xpq$ the triangle formed by (the image of) $\gamma,\xi$ and $\zeta$.
We can construct a map $T:\triangle xpq \lra Y$ to a tripod $(Y,d_Y)$
with three edges of lengths $(p|q)_x$, $(x|q)_p$ and $(x|p)_q$ from the branching point
such that the restrictions $T|_{\gamma}$, $T|_{\xi}$ and $T|_{\zeta}$ are isometric
(see Figure~\ref{fig:tripod}, where $T(a)=T(b)=T(c)$ is the branching point $O$ of the tripod).
We set $\tilde{x}:=T(x)$, $\tilde{p}:=T(p)$ and $\tilde{q}:=T(q)$.
Then $T$ is $1$-Lipschitz (non-expanding) and
\[ d_Y\big( T(u),T(v) \big) \ge d(u,v) -4\delta \]
holds for all $u,v \in \triangle xpq$ by the triangle inequality and the tripod lemma
(see, for instance, \cite[2.15]{Va}, \cite[Lemma~2.3]{Ohyp}).
Together with the Busemann NPC of $(Y,d_Y)$, we deduce that
\[ d\big( \gamma(t),\xi(t) \big)
 \le d_Y\big( T\big( \gamma(t) \big),T\big( \xi(t) \big) \big) +4\delta
 \le td_Y \big( \tilde{p},\tilde{q} \big) +4\delta
 =td(p,q) +4\delta. \]
We similarly obtain $d(\xi(t),\eta(t)) \le (1-t)d(x,y) +4\delta$,
and hence the triangle inequality yields \eqref{eq:Buse}.
\end{proof}

\begin{figure}
\centering
\begin{picture}(380,130)

\put(210,70){$T$}
\put(204,58){$\lra$}

\put(78,67){\rule{2pt}{2pt}}
\put(102,19){\rule{2pt}{2pt}}
\put(127,62){\rule{2pt}{2pt}}

\put(132,66){$a$}
\put(100,5){$b$}
\put(70,72){$c$}

\thicklines
\qbezier(12,10)(92,30)(175,10)
\qbezier(12,10)(62,40)(112,110)
\qbezier(112,110)(122,50)(175,10)

\put(0,5){$x$}
\put(110,119){$p$}
\put(181,6){$q$}

\put(65,94){$(x|q)_p$}
\put(24,47){$(p|q)_x$}


\put(250,10){\line(1,1){50}}
\put(300,60){\line(0,1){50}}
\put(300,60){\line(1,0){60}}

\put(285,62){$O$}
\put(299,59){\rule{2pt}{2pt}}

\put(238,3){$\tilde{x}$}
\put(298,119){$\tilde{p}$}
\put(366,58){$\tilde{q}$}

\put(278,24){$(p|q)_x$}
\put(305,95){$(x|q)_p$}
\put(330,45){$(x|p)_q$}

\end{picture}
\caption{A $1$-Lipschitz map from a triangle to a tripod}\label{fig:tripod}
\end{figure}
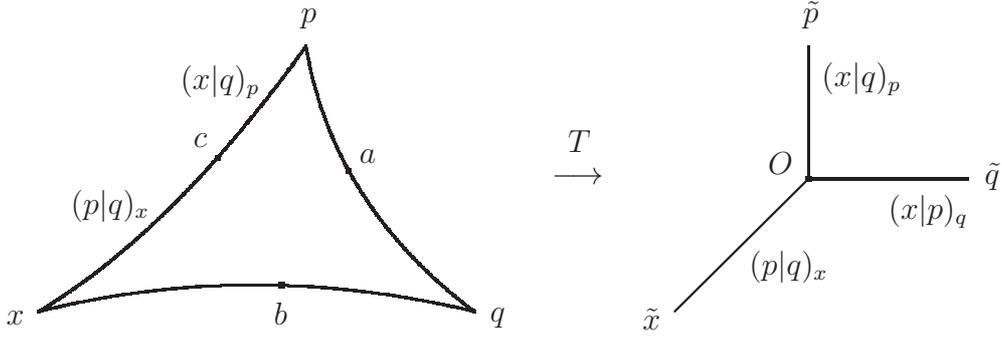

We remark that, similarly to \eqref{eq:CAT0t} in Lemma~\ref{lm:CAT0t},
the inequality \eqref{eq:Buse} does not give a meaningful estimate for $t$ close to $0$ or $1$.
One can use a $1$-Lipschitz map to a tripod also for showing a variant of the $\CAT(0)$-inequality,
whereas then the additional term seems to be necessarily dependent on the size of a triangle
(as in \eqref{eq:CAT0t}), since we take the square of the distance.

Now, let $\sd$ be the $L^2$-distance function on $X \times X \times \R$, namely
\[ \sd\big( (x,y,r),(p,q,s) \big) :=\sqrt{d^2(x,p) +d^2(y,q) +(r-s)^2}. \]
The following subset will play a role:
\[ A :=\{ (x,y,r) \in X \times X \times \R \,|\, d(x,y) \le r \}. \]
If $(X,d)$ is a $\CAT(0)$-space, then $A$ is a (geodesically) convex set thanks to the Busemann NPC.

We will make use of the nearest point projection to $A$ to prove the Wasserstein contraction property.
We remark that $(X \times X \times \R,\sd)$ is not a Gromov hyperbolic space
(it is a $\CAT(0)$-space if so is $(X,d)$), thereby
the contraction property of projection maps as in \cite[Lemma~7.3.D]{Gr} does not directly apply.

\begin{lemma}\label{lm:dA}
For $(x,y,r) \in (X \times X \times [0,\infty)) \setminus A$, we have
\[ \sd\big( (x,y,r),A \big) =\frac{d(x,y)-r}{\sqrt{3}}. \]
\end{lemma}

\begin{proof}
We first show $\sd((x,y,r),A) \ge (d(x,y)-r)/\sqrt{3}$.
Given $(p,q,s) \in A$, if $s<r$, then $(p,q,r) \in A$ and
\[ \sd\big( (x,y,r),(p,q,r) \big) < \sd\big( (x,y,r),(p,q,s) \big). \]
Hence, to show $\sd((x,y,r),(p,q,s)) \ge (d(x,y)-r)/\sqrt{3}$,
we can assume $s \ge r$ without loss of generality
(otherwise we replace $(p,q,s)$ with $(p,q,r)$).
Then we have, together with the triangle inequality,
\begin{align*}
\sd\big( (x,y,r),(p,q,s) \big)
&= \sqrt{d^2(x,p) +d^2(y,q) +(r-s)^2} \\
&\ge \frac{1}{\sqrt{3}} \big\{ d(x,p) +d(y,q) +(s-r) \big\} \\
&\ge \frac{1}{\sqrt{3}} \big\{ d(x,y) -d(p,q) +(s-r) \big\} \\
&\ge \frac{1}{\sqrt{3}} \big\{ d(x,y) -r \}.
\end{align*}
This implies $\sd((x,y,r),A) \ge (d(x,y)-r)/\sqrt{3}$.

To see the reverse inequality, 
we consider a minimal geodesic $\gamma:[0,1] \lra X$ from $x$ to $y$
and put
\[ (p,q,s) :=\bigg( \gamma\bigg( \frac{\lambda}{d(x,y)} \bigg),
 \gamma\bigg( 1-\frac{\lambda}{d(x,y)} \bigg), r+\lambda \bigg),
 \quad \lambda :=\frac{d(x,y)-r}{3} \in \bigg( 0,\frac{d(x,y)}{3} \bigg]. \]
Then $(p,q,s) \in A$ since
\[ d(p,q) =d(x,y) -2\lambda =\frac{d(x,y)+2r}{3} =s, \]
and we have $\sd((x,y,r),(p,q,s)) =\sqrt{3}\lambda$.
This completes the proof.
\end{proof}

The next lemma is an essential step for our contraction result.

\begin{lemma}\label{lm:Wcont}
For $(x,y,r) \in (X \times X \times [0,\infty)) \setminus A$ such that
\begin{equation}\label{eq:ass}
\sd\big( (x,y,r),A \big) \ge \frac{8}{\sqrt{3}}\delta,
\end{equation}
let $(\tilde{x},\tilde{y},\tilde{r}) \in A$ be a point attaining $\sd((x,y,r),A)$
given as in the proof of Lemma~$\ref{lm:dA}$.
Then, for any $(p,q,d(p,q)) \in A$, we have
\begin{align*}
&\sd^2\big( (\tilde{x},\tilde{y},\tilde{r}),(p,q,d(p,q)) \big) \\
&\le \sd^2\big( (x,y,r),(p,q,d(p,q)) \big)
 -\sd^2\big( (x,y,r),A \big) +18D_1 \sqrt{D_1 +\delta} \cdot \sqrt{\delta},
\end{align*}
where we set
\[ D_1=D_1(x,y,p,q) :=\diam\big( (\{x,y,p,q\},d) \big). \]
\end{lemma}

Observe that, if $\delta=0$, then the assumption \eqref{eq:ass} is void
and the assertion shows that $\sd((\tilde{x},\tilde{y},\tilde{r}),(p,q,d(p,q)))$
is strictly smaller than $\sd( (x,y,r),(p,q,d(p,q)))$.

\begin{proof}
Put $s:=d(p,q)$ for brevity, and
let $\gamma:[0,1] \lra X$ and $\eta:[0,1] \lra X$ be minimal geodesics
from $\tilde{x}$ to $p$ and from $\tilde{y}$ to $q$, respectively.
Note that Lemma~\ref{lm:Buse} yields
\[ d\big( \gamma(t),\eta(t) \big)
 \le (1-t)d(\tilde{x},\tilde{y}) +td(p,q) +8\delta \]
for any $t \in [0,1]$.
Then we deduce from Lemma~\ref{lm:dA} that
\begin{align*}
&\sd \Big( \big( \gamma(t),\eta(t),(1-t)\tilde{r}+ts \big),A \Big) \\
&\le \frac{1}{\sqrt{3}} \Big\{ (1-t)\big( d(\tilde{x},\tilde{y})-\tilde{r} \big) +t\big( d(p,q)-s \big) \Big\}
 +\frac{8\delta}{\sqrt{3}} \\
&= \frac{8}{\sqrt{3}}\delta.
\end{align*}
Therefore, on the one hand,
\begin{align*}
&\sd\Big( (x,y,r), \big( \gamma(t),\eta(t),(1-t)\tilde{r}+ts \big) \Big) \\
&\ge \sd\big( (x,y,r),A \big)
 -\sd\Big( \big( \gamma(t),\eta(t),(1-t)\tilde{r}+ts \big),A \Big) \\
&\ge \sd\big( (x,y,r),A \big) -\frac{8}{\sqrt{3}}\delta.
\end{align*}
This implies that, since the right-hand side is nonnegative by the assumption \eqref{eq:ass}
and $\sd((x,y,r),A) \le D_1/\sqrt{3}$ by Lemma~\ref{lm:dA} and $r \ge 0$,
\begin{equation}\label{eq:dA>}
\sd^2 \Big( (x,y,r), \big( \gamma(t),\eta(t),(1-t)\tilde{r}+ts \big) \Big)
 \ge \sd^2 \big( (x,y,r),A \big) -\frac{16}{3}D_1 \delta +\frac{64}{3}\delta^2. 
\end{equation}

On the other hand, observe from \eqref{eq:CAT0t} that
\begin{align*}
d^2\big( x,\gamma(t) \big)
&\le (1-t)d^2(x,\tilde{x}) +td^2(x,p) -(1-t)td^2(\tilde{x},p) +4D_1 \delta +4\delta^2, \\
d^2\big( y,\eta(t) \big)
&\le (1-t)d^2(y,\tilde{y}) +td^2(y,q) -(1-t)td^2(\tilde{y},q) +4D_1 \delta +4\delta^2
\end{align*}
(recall that $\tilde{x}$ and $\tilde{y}$ are on a minimal geodesic between $x$ and $y$).
Summing up, we obtain
\begin{align*}
&\sd^2\Big( (x,y,r), \big( \gamma(t),\eta(t),(1-t)\tilde{r}+ts \big) \Big) \\
&\le (1-t)\sd^2\big( (x,y,r),(\tilde{x},\tilde{y},\tilde{r}) \big)
 +t\sd^2\big( (x,y,r),(p,q,s) \big) \\
&\quad -(1-t)t\sd^2\big( (\tilde{x},\tilde{y},\tilde{r}),(p,q,s) \big)
 +8D_1 \delta +8\delta^2.
\end{align*}
Combining this with \eqref{eq:dA>} shows
\begin{align*}
&t\sd^2\big( (x,y,r),(p,q,s) \big) -(1-t)t\sd^2\big( (\tilde{x},\tilde{y},\tilde{r}),(p,q,s) \big) \\
&\ge t\sd^2\big( (x,y,r),A \big) -\frac{40}{3}D_1 \delta +\frac{40}{3}\delta^2.
\end{align*}
By rearrangement, we find
\begin{align*}
\sd^2\big( (\tilde{x},\tilde{y},\tilde{r}),(p,q,s) \big)
&\le \sd^2\big( (x,y,r),(p,q,s) \big)
 -\sd^2\big( (x,y,r),A \big) \\
&\quad +t\sd^2\big( (\tilde{x},\tilde{y},\tilde{r}),(p,q,s) \big)
 +\frac{40}{3}\frac{D_1 \delta}{t}.
\end{align*}
Moreover, it follows from $d(x,\tilde{x}) =d(y,\tilde{y}) \le d(x,y)/3$,
$\tilde{r} =d(\tilde{x},\tilde{y}) <d(x,y)$ and $s=d(p,q)$ that
\[ \sd^2 \big( (\tilde{x},\tilde{y},\tilde{r}),(p,q,s) \big)
 \le \frac{16}{9}D_1^2 +\frac{16}{9}D_1^2 +D_1^2
 =\frac{41}{9}D_1^2. \]
Finally, letting $t =\sqrt{\delta/(D_1+\delta)} \in (0,1)$, we obtain
\begin{align*}
&\sd^2\big( (\tilde{x},\tilde{y},\tilde{r}),(p,q,s) \big) \\
&\le \sd^2\big( (x,y,r),(p,q,s) \big) -\sd^2\big( (x,y,r),A \big)
 +\bigg( \frac{41D_1^2}{9\sqrt{D_1 +\delta}} +\frac{40}{3}D_1 \sqrt{D_1 +\delta} \bigg) \sqrt{\delta} \\
&\le \sd^2\big( (x,y,r),(p,q,s) \big) -\sd^2\big( (x,y,r),A \big)
 +18D_1 \sqrt{D_1 +\delta} \cdot \sqrt{\delta}.
\end{align*}
\end{proof}

We remark that, as is natural from the local flexibility of scale $\le \delta$,
the following contraction property is nontrivial
only when $\delta$ is sufficiently small ($\delta \ll D_2$),
and then we have $d(x,y) \le W_1(\mu,\nu) +O(\delta^{1/4})$ for $\ve_1 =\ve_2 =0$.

\begin{theorem}[Wasserstein contraction]\label{th:Wcont}
For any $\mu,\nu \in \cP^1(X)$, $x \in \cB(\mu,\ve_1)$ and $y \in \cB(\nu,\ve_2)$
with $\ve_1,\ve_2 \ge 0$, we have
\[ d(x,y) \le W_1(\mu,\nu)
 +8\delta \vee \sqrt{54D_2 \sqrt{D_2 +\delta} \sqrt{\delta} +3(\ve_1 +\ve_2)}, \]
where
\[ D_2 =D_2(x,y,\mu,\nu) :=\diam\big( (\{x,y\} \cup \supp\,\mu \cup \supp\,\nu,d) \big). \]
\end{theorem}

\begin{proof}
For arbitrary $\alpha>0$,
let $\pi \in \cP(X \times X)$ be a coupling of $(\mu,\nu)$ with
\[ r:= \int_{X \times X} d(p,q) \,\pi(dp\,dq) \le W_1(\mu,\nu) +\alpha. \]
We define a map
\[ \Phi:X \times X \ni (p,q) \,\longmapsto\, \big( p,q,d(p,q) \big) \in X \times X \times \R, \]
and put $\Pi :=\Phi_* \pi$ (the push-forward of $\pi$ by $\Phi$).
Note that $\supp\,\Pi \subset \Phi(\supp\,\pi) \subset A$.
Since $x \in \cB(\mu,\ve_1)$ and $y \in \cB(\nu,\ve_2)$,
together with the choice of $r$, we find
\begin{align*}
&\int_{X \times X \times \R}
 \big\{ \sd^2 \big( (x,y,r),(p,q,s) \big) -d^2(x_0,p) -d^2(x_0,q) -s^2 \big\} \,\Pi(dp\,dq\,ds) \\
&= \int_X \big\{ d^2(x,p) -d^2(x_0,p) \big\} \,\mu(dp) +\int_X \big\{ d^2(y,q) -d^2(x_0,q) \big\} \,\nu(dq) \\
&\quad +\int_{X \times X} \big\{ \big( r-d(p,q) \big)^2 -d^2(p,q) \big\} \,\pi(dp\,dq) \\
&\le \inf_{(\bar{x},\bar{y},\bar{r}) \in X \times X \times \R}
 \int_{X \times X \times \R}
 \big\{ \sd^2\big( (\bar{x},\bar{y},\bar{r}),(p,q,s) \big) -d^2(x_0,p) -d^2(x_0,q) -s^2 \big\} \,\Pi(dp\,dq\,ds) \\
&\quad +\ve_1 +\ve_2.
\end{align*}
That is to say, $(x,y,r)$ lives in the barycentric set $\cB(\Pi,\ve_1 +\ve_2)$ in $(X \times X \times \R,\sd)$.

Then we deduce from Lemma~\ref{lm:Wcont} that $(x,y,r)$ necessarily satisfies
\[ \sd^2\big( (x,y,r),A \big) \le \frac{64}{3}\delta^2 \vee
 \bigg( 18D_2 \sqrt{D_2 +\delta} \sqrt{\delta} +\ve_1 +\ve_2 \bigg). \]
Indeed, if not, then \eqref{eq:ass} is fulfilled and, for any $(p,q) \in \supp\,\pi$,
$(\tilde{x},\tilde{y},\tilde{r}) \in A$ (given as in the proof of Lemma~\ref{lm:dA}) satisfies
\begin{align*}
&\sd^2\big( (\tilde{x},\tilde{y},\tilde{r}),(p,q,d(p,q)) \big) \\
&\le \sd^2\big( (x,y,r),(p,q,d(p,q)) \big) -\sd^2\big( (x,y,r),A \big)
 +18D_2 \sqrt{D_2 +\delta} \sqrt{\delta} \\
&< \sd^2\big( (x,y,r),(p,q,d(p,q)) \big) -(\ve_1 +\ve_2).
\end{align*}
By integrating in $(p,q)$ with respect to $\pi$,
we find a contradiction to $(x,y,r) \in \cB(\Pi,\ve_1 +\ve_2)$.

Hence, it follows from Lemma~\ref{lm:dA} that
\begin{align*}
d(x,y) -W_1(\mu,\nu) &\le d(x,y) -r +\alpha \\
&\le 8\delta \vee
 \sqrt{54D_2 \sqrt{D_2 +\delta} \sqrt{\delta} +3(\ve_1 +\ve_2)} +\alpha.
\end{align*}
Letting $\alpha \to 0$ completes the proof.
\end{proof}

\section{Deterministic approximations of barycenters}\label{sc:nodice}

We next discuss an approximation of barycenters of uniform distributions on finite points
by the gradient flow method.
Given a function $f:X \lra \R$,
we define the \emph{proximal operator} (also called the \emph{resolvent operator}) by
\begin{equation}\label{eq:sJ}
\sJ^f_{\tau}(x) :=\argmin_{y \in X} \bigg\{ f(y) +\frac{d^2(x,y)}{2\tau} \bigg\}
\end{equation}
for $x \in X$ and $\tau>0$ (that is, $y \in \sJ^f_{\tau}(x)$ if $y$ attains the above minimum).
Roughly speaking, an element in $\sJ^f_{\tau}(x)$ can be regarded as
an approximation of a point on the gradient curve of $f$ at time $\tau$ from $x$.
Thus the iteration of the proximal operator can be regarded as \emph{discrete-time gradient flow} for $f$,
which is expected to lead us to a minimizer of $f$.
We refer to \cite{Ohyp} for some contraction properties of discrete-time gradient flows
for geodesically convex functions on Gromov hyperbolic spaces.

We will apply the proximal operator only to squared distance functions $f=d^2(z,\cdot)$.
In this case, each $y \in \sJ^f_{\tau}(x)$ is explicitly given as $y=\gamma((2\tau +1)^{-1})$
for some minimal geodesic $\gamma:[0,1] \lra X$ from $z$ to $x$.

The next lemma (corresponding to \cite[Lemma~3.2]{Ba2}, \cite[(4.5)]{OP1}) provides a key estimate.
Precisely, those in \cite{Ba2,OP1} are concerned with $y \in \sJ^f_{\tau}(x)$ for a convex function $f$,
and we shall generalize it to Gromov hyperbolic spaces for the squared distance function
$f=d^2(z,\cdot)$.

\begin{lemma}[Key estimate]\label{lm:y<x}
For any $\tau>0$, $w,x,z \in X$ and $y=\gamma((2\tau+1)^{-1})$
on a minimal geodesic $\gamma:[0,1] \lra X$ from $z$ to $x$, we have
\begin{equation}\label{eq:y<x}
d^2(w,y) \le d^2(w,x) -2\tau\big\{ d^2(z,y)-d^2(z,w) \big\} +\Theta\tau\delta,
\end{equation}
where
\[ \Theta :=\big\{ 8d(z,w) +8\delta \big\} \vee
 \bigg[ \big\{ 4d(w,y)+8\tau d(z,y) \big\} \bigg( \frac{1}{\tau} \wedge \frac{2d(z,y)}{\delta} \bigg) \bigg]. \]
\end{lemma}

\begin{proof}
First of all, it follows from $y \in \sJ^f_{\tau}(x)$ with $f=d^2(z,\cdot)$ that
\[ d^2(x,y) \le d^2(x,w) -2\tau\big\{ d^2(z,y)-d^2(z,w) \big\}. \]
Hence, if $d(w,y) \le d(x,y)$, then \eqref{eq:y<x} holds
(without the additional term $\Theta\tau\delta$).

We assume $d(w,y)>d(x,y)$ in the sequel.
Note that the $\delta$-hyperbolicity implies
\[ 0 =(x|z)_y \ge (x|w)_y \wedge (z|w)_y -\delta. \]
If $(z|w)_y \le \delta$, then we find
\[ \big\{ d(z,w) +2\delta \big\}^2
 \ge \big\{ d(z,y)+d(w,y) \big\}^2
 \ge d^2(z,y) +2d(z,y) d(w,y). \]
Combining this with the triangle inequality $|d(w,y)-d(x,y)| \le d(w,x)$
and $d(x,y) =2\tau d(z,y)$, we obtain
\begin{align*}
d^2(w,y) &\le d^2(w,x) +2d(x,y) d(w,y) -d^2(x,y) \\
&\le d^2(w,x) +4\tau d(z,y) d(w,y)  \\
&\le d^2(w,x) +2\tau \big\{ d^2(z,w) -d^2(z,y) \big\}
 +2\tau \big\{ 4\delta d(z,w) +4\delta^2 \big\}.
\end{align*}

In the other case of $(x|w)_y \le \delta$,
together with the triangle inequality, we have
\[ d(x,y) +d(w,y) -d(w,x) \le (2\delta) \wedge \big( 2d(x,y) \big). \]
Thus we obtain
\[ d(w,x) \ge d(w,y) +d(x,y) -2\big( \delta \wedge d(x,y) \big), \]
and observe that the right-hand side is positive by the hypothesis $d(w,y)>d(x,y)$.
Therefore, we deduce that
\[ d^2(w,x) \ge d^2(w,y) +d^2(x,y) +2d(w,y) d(x,y)
 -4\big\{ d(w,y)+d(x,y) \big\} \big( \delta \wedge d(x,y) \big). \]
Substituting $d(x,y)=2\tau d(z,y)$ and $d(w,y) \ge d(z,y)-d(z,w)$ yields
\begin{align*}
d^2(w,x) &\ge d^2(w,y) +4\tau(\tau +1)d^2(z,y) -4\tau d(z,w) d(z,y) \\
&\quad -4\big\{ d(w,y)+d(x,y) \big\} \big( \delta \wedge 2\tau d(z,y) \big).
\end{align*}
Then we apply the elementary inequality
$2(\tau +1)a^2 -2ab \ge a^2 -b^2$ with $a=d(z,y)$ and $b=d(z,w)$ to see
\[ d^2(w,x) \ge d^2(w,y) +2\tau \big\{ d^2(z,y) -d^2(z,w) \big\}
 -4\big\{ d(w,y)+d(x,y) \big\} \big( \delta \wedge 2\tau d(z,y) \big). \]
This completes the proof.
\end{proof}

In the $\CAT(0)$-case (see \cite{Ba2,OP1}), we have
\[ d^2(w,y) \le d^2(w,x) -2\tau\big\{ d^2(z,y)-d^2(z,w) \big\} \]
without the additional term $\Theta\tau\delta$.
Note that $\Theta\tau\delta$ in \eqref{eq:y<x} tends to $0$
not only as $\delta \to 0$ but also as $\tau \to 0$.
This is the natural behavior since $y$ tends to $x$ as $\tau \to 0$.

We also remark that, in the proof of \cite[(4.5)]{OP1},
the Riemannian nature of $\CAT(0)$-spaces plays an essential role.
Precisely, the inequality \cite[(4,2)]{OP1} is a consequence of
a Riemannian property called the \emph{commutativity} as in \cite[(1.2), (3.1)]{OP2}.

By using Lemma~\ref{lm:y<x}, we establish the following deterministic approximation of barycenters
by the iterative application of the proximal operator in the spirit of \cite{Ho,LP}
(so-called the \emph{no dice} theorem).
We also refer to \cite[Theorem~3.4]{Ba2} and \cite[Theorem~5.5]{OP1}
for generalizations to the sum of convex functions.


\begin{theorem}[Deterministic approximation]\label{th:nodice}
Fix a finite sequence $(z_i)_{i=1}^n$ in $X$, put $f_i:=d^2(z_i,\cdot)$,
and let $p \in X$ be a minimizer of the function $f:=\sum_{i=1}^n f_i$.
Given $\tau>0$ and an arbitrary initial point $y_0 \in X$, we recursively choose
\[ y_{kn+i} \in \sJ^{f_i}_{\tau}(y_{kn+i-1})
 \quad\ \text{for}\,\ k \ge 0,\ 1 \le i \le n, \]
and assume that $p$, $(z_i)$ and $(y_{kn+i})$ are all included in a bounded set $\Omega \subset X$.
Then, for any $\ve>0$, there exists some $k_0 <d^2(p,y_0)/(2\tau\ve)$ such that
\begin{equation}\label{eq:nod-1}
f(y_{k_0 n}) \le f(p) +\frac{n\Theta_{\Omega} \delta}{2} +2n(n+1)D_{\Omega}^2 \tau +\ve,
\end{equation}
where we set $D_{\Omega}:=\diam(\Omega)$ and
\[ \Theta_{\Omega} :=(8D_{\Omega} +8\delta) \vee
 \bigg[ (4+8\tau)D_{\Omega} \bigg( \frac{1}{\tau} \wedge \frac{2D_{\Omega}}{\delta} \bigg) \bigg]. \]
Moreover, we have
\begin{equation}\label{eq:nod-2}
d(p,y_{k_0 n}) \le \sqrt{ (16D_{\Omega} +\Theta_{\Omega})\delta +16\delta^2
 +4(n+1)D_{\Omega}^2 \tau +\frac{2\ve}{n}}.
\end{equation}
\end{theorem}

We remark that, in a $\CAT(0)$-space,
we can employ as $\Omega$ a ball including $(z_i)_{i=1}^n$ and $y_0$.
This is because balls are convex by the $\CAT(0)$-inequality (or the Busemann NPC).
In a $\delta$-hyperbolic space, however, balls are not necessarily convex
and it is unclear to the author how to control
(the sum of) the additional terms in \eqref{eq:CAT0t} (or \eqref{eq:Buse})
during the recursive scheme $y_{kn+i} \in \sJ^{f_i}_{\tau}(y_{kn+i-1})$.

\begin{proof}
It follows from Lemma~\ref{lm:y<x} that
\[ d^2(p,y_{kn+i}) \le d^2(p,y_{kn+i-1})
 -2\tau \big\{ d^2(z_i,y_{kn+i}) -d^2(z_i,p) \big\} +\Theta_{\Omega} \tau\delta. \]
Summing up for $1 \le i \le n$, we have
\begin{align}
&d^2(p,y_{(k+1)n}) \nonumber\\
&\le d^2(p,y_{kn}) -2\tau \big\{ f(y_{kn}) -f(p) \big\}
 +2\tau \bigg\{ f(y_{kn}) -\sum_{i=1}^n d^2(z_i,y_{kn+i}) \bigg\} +n\Theta_{\Omega} \tau\delta.
 \label{eq:nodice}
\end{align}
Concerning the third term in the right-hand side, we infer from the triangle inequality that
\begin{align*}
\sum_{i=1}^n \big\{ d^2(z_i,y_{kn}) -d^2(z_i,y_{kn+i}) \big\}
&= \sum_{i=1}^n \sum_{l=1}^i \big\{ d^2(z_i,y_{kn+l-1}) -d^2(z_i,y_{kn+l}) \big\} \\
&\le 2D_{\Omega} \sum_{i=1}^n \sum_{l=1}^i d(y_{kn+l-1},y_{kn+l}).
\end{align*}
Then, by the choice of $y_{kn+l} \in \sJ^{f_l}_{\tau}(y_{kn+l-1})$ with $f_l=d^2(z_l,\cdot)$,
we find
\begin{equation}\label{eq:kn+l}
d(y_{kn+l-1},y_{kn+l}) =\frac{2\tau}{2\tau +1} d(z_l,y_{kn+l-1})
 \le 2D_{\Omega} \tau.
\end{equation}
Hence, we obtain
\[ \sum_{i=1}^n \big\{ d^2(z_i,y_{kn}) -d^2(z_i,y_{kn+i}) \big\}
 \le 4D_{\Omega}^2 \tau \sum_{i=1}^n i =2n(n+1)D_{\Omega}^2 \tau. \]
Plugging this into \eqref{eq:nodice} yields
\begin{equation}\label{eq:k+1}
d^2(p,y_{(k+1)n}) \le d^2(p,y_{kn}) 
 -2\tau \bigg\{ f(y_{kn}) -f(p) -\frac{n\Theta_{\Omega} \delta}{2}
 -2n(n+1)D_{\Omega}^2 \tau \bigg\}.
\end{equation}

It immediately follows from \eqref{eq:k+1} that
\[ f(y_{k_0 n}) -f(p) -\frac{n\Theta_{\Omega} \delta}{2} -2n(n+1)D_{\Omega}^2 \tau \le \ve \]
necessarily holds for some $k_0 <d^2(p,y_0)/(2\tau\ve)$.
Indeed, otherwise we have
\[ d^2(p,y_{\bar{k}n}) < d^2(p,y_0) -2\tau\bar{k}\ve \le 0 \]
with $\bar{k}$ the minimum integer not smaller than $d^2(p,y_0)/(2\tau\ve)$, a contradiction.

Finally, the second assertion \eqref{eq:nod-2} is a consequence of Proposition~\ref{pr:bary}.
Putting $\mu =n^{-1} \sum_{i=1}^n \delta_{z_i}$,
we deduce from \eqref{eq:nod-1} that
\[ W_2^2(\delta_{y_{k_0 n}},\mu) \le W_2^2(\delta_p,\mu)
 +\frac{\Theta_{\Omega}\delta}{2} +2(n+1)D_{\Omega}^2 \tau +\frac{\ve}{n}. \]
Hence, \eqref{eq:vineq} (with $\ve_1=0$) yields
\begin{align*}
d(p,y_{k_0 n}) &\le
 \sqrt{8\delta \big\{ W_1(\delta_p,\mu) +W_1(\delta_{y_{k_0 n}},\mu) \big\} +16\delta^2
 +\Theta_{\Omega}\delta +4(n+1)D_{\Omega}^2 \tau +\frac{2\ve}{n}} \\
&\le \sqrt{16D_{\Omega} \delta +16\delta^2
 +\Theta_{\Omega}\delta +4(n+1)D_{\Omega}^2 \tau +\frac{2\ve}{n}}.
\end{align*}
\end{proof}

Thanks to \eqref{eq:nod-1}, up to $d^2(p,y_0)/(2\tau\ve)$ iterations,
$f(y_{kn})$ nearly achieves $\min_X f$ and $y_{kn}$ passes close to $p$ as in \eqref{eq:nod-2}.
We remark that the sublinear rate $\ve < d^2(p,y_0)/(2\tau k_0)$
(following from $k_0 < d^2(p,y_0)/(2\tau\ve)$) could be compared with \cite[Proposition~5.7]{OP1}.

Note also that we have an effective estimate on
the resolvent operator only when $\tau$ is larger than $\delta$
(so-called ``giant steps''; see \cite{Ohyp} for a further discussion),
thereby we did not consider $(\tau_k)_{k \ge 0}$ converging to $0$
(compare this with \cite[Theorem~5.5]{OP1}).
If we assume $\delta \le D_{\Omega}/2$ and take $\tau =\sqrt{\delta/D_{\Omega}}$,
then we have $\Theta_{\Omega} =(4\sqrt{D_{\Omega}/\delta} +8)D_{\Omega}$
and \eqref{eq:nod-2} shows
\[ d(p,y_{k_0 n}) \le \sqrt{\frac{2\ve}{n}} +O(\delta^{1/4}). \]

\section{A law of large numbers}\label{sc:LLN}

We next consider a law of large numbers in our setting.
Our formulation follows Sturm's \cite[Theorem~4.7]{St1} for $\CAT(0)$-spaces.
We refer to \cite[Theorem~6.7]{OP1} and \cite[Theorem~3]{Yo3} for some generalizations
to other (upper and lower) curvature bounds.

\begin{theorem}[Law of large numbers]\label{th:LLN}
Let $(Z_i)_{i \ge 1}$ be a sequence of independent, identically distributed random variables
on a probability space taking values in $X$ with distribution $\mu \in \cP(X)$,
and take $p \in \cB(\mu,0)$.
Given $\tau>0$ and an arbitrary initial point $S_0 \in X$,
we define a sequence $(S_k)_{k \ge 0}$ in $X$ recursively by
\[ S_{k+1} \in \sJ^{f_k}_{\tau}(S_k), \qquad f_k:=d^2(Z_{k+1},\cdot). \]
Assume that $p$, $\supp\,\mu$ and $(S_k)_{k \ge 0}$ are all included in a bounded set $\Omega \subset X$.
Then, for any $\ve>0$, we have
\begin{equation}\label{eq:LLN}
\E\big[ d^2(p,S_{k_0}) \big]
 \le 8D_{\Omega}^2 \tau +(\Theta_{\Omega} +16D_{\Omega} +16\delta)\delta +\ve
\end{equation}
for some $k_0 <d^2(p,S_0)/(\tau\ve)$, where
$D_{\Omega}:=\diam(\Omega)$ and $\Theta_{\Omega}$ is defined as in Theorem~$\ref{th:nodice}$.
\end{theorem}

\begin{proof}
We can apply a calculation similar to the proof of Theorem~\ref{th:nodice}.
It follows from Lemma~\ref{lm:y<x} that
\begin{align*}
d^2(p,S_{k+1})
&\le d^2(p,S_k) -2\tau \big\{ d^2(Z_{k+1},S_{k+1}) -d^2(Z_{k+1},p) \big\} +\Theta_{\Omega} \tau\delta \\
&= d^2(p,S_k) -2\tau \big\{ d^2(Z_{k+1},S_k) -d^2(Z_{k+1},p) \big\} \\
&\quad +2\tau \big\{ d^2(Z_{k+1},S_k) -d^2(Z_{k+1},S_{k+1}) \big\} +\Theta_{\Omega} \tau\delta \\
&\le d^2(p,S_k) -2\tau \big\{ d^2(Z_{k+1},S_k) -d^2(Z_{k+1},p) \big\}
 +8D_{\Omega}^2 \tau^2 +\Theta_{\Omega} \tau\delta,
\end{align*}
where we used
\[ d^2(Z_{k+1},S_k) -d^2(Z_{k+1},S_{k+1})
 \le 2D_{\Omega} d(S_k,S_{k+1}) \le 4D_{\Omega}^2 \tau \]
in the latter inequality (recall also \eqref{eq:kn+l}).
Taking the expectations in $Z_{k+1}$ conditioned on $\mathcal{F}_k :=\{Z_1,\ldots,Z_k\}$
and applying the variance inequality \eqref{eq:vineq}
(with $x=p$, $y=S_k$, $\ve_1=0$), we obtain
\begin{align*}
\E\big[ d^2(p,S_{k+1}) \,\big|\, \mathcal{F}_k \big]
&\le d^2(p,S_k) -2\tau \E\big[ d^2(Z_{k+1},S_k) -d^2(Z_{k+1},p) \big]
 +8D_{\Omega}^2 \tau^2 +\Theta_{\Omega} \delta\tau \\
&\le d^2(p,S_k) -\tau \big\{ d^2(p,S_k) -16D_{\Omega} \delta -16\delta^2 \big\}
 +8D_{\Omega}^2 \tau^2 +\Theta_{\Omega} \delta\tau \\
&= (1-\tau) d^2(p,S_k)
 +8D_{\Omega}^2 \tau^2 +(\Theta_{\Omega} +16D_{\Omega} +16\delta)\delta\tau.
\end{align*}
Taking the expectations once again, we arrive at
\begin{equation}\label{eq:LLN+1}
\E\big[ d^2(p,S_{k+1}) \big]
 \le (1-\tau) \E\big[ d^2(p,S_k) \big]
 +8D_{\Omega}^2 \tau^2 +(\Theta_{\Omega} +16D_{\Omega} +16\delta)\delta\tau.
\end{equation}

In the same way as in the proof of Theorem~\ref{th:nodice},
we infer from \eqref{eq:LLN+1} that 
\[ \E\big[ d^2(p,S_{k_0}) \big]
 \le 8D_{\Omega}^2 \tau +(\Theta_{\Omega} +16D_{\Omega} +16\delta)\delta +\ve \]
necessarily holds for some $k_0 <d^2(p,S_0)/(\tau\ve)$.
This completes the proof.
\end{proof}

When we assume $\delta \le D_{\Omega}/2$ and choose $\tau=\sqrt{\delta/D_{\Omega}}$
as in the discussion after Theorem~\ref{th:nodice}, \eqref{eq:LLN} yields
\[ \E\big[ d^2(p,S_{k_0}) \big] \le \ve +O(\sqrt{\delta}). \]

An advantage of the above recursive choice of $(S_k)_{k \ge 0}$ is that
$S_{k+1}$ is concretely given as a point on a geodesic between $S_k$ and $Z_{k+1}$
without any knowledge about the construction of barycenters of probability measures
(though minimal geodesics are not unique in Gromov hyperbolic spaces).

Employing ``empirical means'' instead of $(S_k)_{k \ge 0}$,
one can also show the following version of law of large numbers (see \cite[Proposition~6.6]{St1}).

\begin{proposition}[Empirical law of large numbers]\label{pr:LLN}
Let $(X,d)$ be complete and separable,
$(Z_i)_{i \ge 1}$ be a sequence of independent, identically distributed random variables
on a probability space taking values in $X$ with distribution $\mu \in \cP(X)$ of bounded support,
and take $p \in \cB(\mu,0)$.
Then,
\[ \sigma_k \in \cB\Bigg( \frac{1}{k} \sum_{i=1}^k \delta_{Z_i},0 \Bigg) \]
satisfies
\[ \limsup_{k \to \infty} d(p,\sigma_k)
 \le 8\delta \vee \sqrt{54D \sqrt{D+\delta} \sqrt{\delta}} \]
almost surely, where we set $D:=3\,\diam(\supp\,\mu)$.
\end{proposition}

\begin{proof}
It follows from Varadarajan's theorem (see \cite[Theorem~11.4.1]{Du}) that
\[ \nu_k :=\frac{1}{k} \sum_{i=1}^k \delta_{Z_i} \]
weakly converges to $\mu$ almost surely,
which implies $W_1(\mu,\nu_k) \to 0$ (see \cite[Theorem~7.12]{Vi}).
Observe that $d(p,\supp\,\mu) \le \diam(\supp\,\mu)$ necessarily holds and similarly
\[ d(\sigma_k,\supp\,\mu) \le d(\sigma_k,\supp\,\nu_k)
 \le \diam(\supp\,\nu_k) \le \diam(\supp\,\mu). \]
Hence, we have $\diam((\{p,\sigma_k\} \cup \supp\,\mu,d)) \le D$ and
Theorem~\ref{th:Wcont} (with $\ve_1=\ve_2=0$) yields
\[ \limsup_{k \to \infty} d(p,\sigma_k)
 \le 8\delta \vee \sqrt{54D \sqrt{D+\delta} \sqrt{\delta}} \]
by the choices of $p$ and $\sigma_k$.
\end{proof}

We finally discuss two possible directions of further research.

\begin{remark}[Further problems]\label{rm:final}
\begin{enumerate}[(a)]
\item
We have studied in \cite{Ohyp} discrete-time gradient flows for geodesically convex functions
(namely, they are convex along geodesics).
However, despite the negative curvature nature, distance functions on Gromov hyperbolic spaces
are not geodesically convex due to the local flexibility.
Therefore, it is an intriguing problem to introduce an appropriate notion of ``roughly convex functions''
on Gromov hyperbolic spaces, including the distance function $d(z,\cdot)$ or its square.

\item
Compared with the contraction estimates in \cite{Ohyp} directly akin to trees,
the results in the present article are generalizations from $\CAT(0)$-spaces to $\delta$-hyperbolic spaces.
Therefore, on the one hand, it may be possible to improve our estimates (for example, the order of $\delta$)
via an analysis closer to the case of trees.
Specifically, it is desirable that we can reduce the dependence on $D_2,D_{\Omega}$
in Theorems~\ref{th:Wcont}, \ref{th:nodice}, \ref{th:LLN}.
On the other hand, there seems a room for further generalizations
to metric spaces satisfying the $\CAT(0)$-inequality with small perturbations in some way
(cf.\ Lemma~\ref{lm:CAT0t}).
\end{enumerate}
\end{remark}

\textit{Acknowledgements.}
This work was supported in part by JSPS Grant-in-Aid for Scientific Research (KAKENHI)
19H01786, 22H04942.

{\small

}

\end{document}